\documentclass{article}%
\usepackage{amsmath}
\usepackage{amsfonts}
\usepackage{amssymb}
\usepackage{graphicx}%
\providecommand{\U}[1]{\protect\rule{.1in}{.1in}}
\newtheorem{theorem}{Theorem}

\newtheorem{corollary}[theorem]{Corollary}

\newtheorem{definition}[theorem]{Definition}

\newtheorem{lemma}[theorem]{Lemma}

\newtheorem{remark}[theorem]{Remark}

\newenvironment{proof}[1][Proof]{\noindent \textbf{#1.} }{\  \rule{0.5em}{0.5em}}
\begin{document}

\title{Note on Viscosity Solution of Path-Dependent PDE and $G$-Martingales --- 2nd version}
\author{Shige Peng\\School of Mathematics, Shandong University}
\date{February 19, 2012}
\maketitle

\abstract{In this note we introduce the notion of viscosity solution
for a type of fully nonlinear parabolic path-dependent partial
differential equations (P-PDE).  We obtain new maximum principles,
(or comparison theorem) for smooth solutions as well as for
viscosity solutions. A solution of a backward stochastic
differential equation and a $G$-martingale under a $G$-expectation
are typical examples of such type of solutions of P-PDE.  }

\medskip\medskip


\medskip


\section{Introduction}

In general, a solution of a backward stochastic differential equation (BSDE in
short) is an adapted process, namely, the value of this process at each time
$t$ is a functional of the corresponding continuous path on $[0,t]$. When this
value depends only on the current state of the path $\omega(t)$, we have
proved (see [Peng1991], [Peng1992a,b], and [Peng-Pardoux1992]) that the
solution of the BSDE is in fact a solution of a quasi-linear parabolic PDE.
This relation was given by introducing what we called nonlinear Feynman-Kac
formula. Recently we have introduced a new notion of $G$-martingale under a
fully nonlinear expectation called $G$-expectation. A $G$-martingale can be
also regarded as a solution of fully nonlinear BSDE if the solution is state
dependent. We also refer to the 2BSDE formulation for such second order
nonlinear BSDE (see \cite{CSTV2007} and \cite{STZh2009-2010}). It is then a
very interesting problem, which was proposed in my lecture of ICM2010
[Peng2010b], whether a path dependent solution of a BSDE and/or a
$G$-martingale can be considered as a nonlinear path-dependent PDE (PPDE) of
parabolic and/or elliptic types.

Facing this challenge, in this note we will introduce a notion of
viscosity solution for the above mentioned types of quasi-linear or
fully nonlinear PPDE. Just as in the case of the classical PDE, an
important advantage of a viscosity solution of PPDE is that we only
need it to be a continuous functional of paths. Here a crucially
important task is to prove the corresponding comparison (also called
maximum) principle, or comparison theorem, which is the main
objective of this note.

Smooth solutions of linear path-dependent PDE of parabolic types were
initially introduced in [Dupire2009] in which a new type of functional It\^{o}
formula was proved and then used to find a $\mathbb{C}^{1,2}$ solution of the
PDE. We also refer to [Cont-Fourni\'{e}2010a,b] for further developments of
this new calculus of It\^{o}'s type with applications to finance. Recently we
have obtained the existence and uniqueness of systems of smooth solutions of
quasi-linear path-dependent PDE by a BSDE approach, see [{Peng-Wang2011].
}These methods are mainly based on stochastic calculus. Our approach in this
note is based on techniques of PDE and can be directly applied to treat fully
nonlinear path-dependent PDE. The advantage of this PDE approach is that, one
can treat the solution locally (path by path), whereas BSDE and $G$%
-expectation are mainly a global approach.

In this 2nd version of the paper, the main improvement is as
follows: It is known that the proof of maximum principles often
involves a maximization procedure of the difference of a subsolution
and a supersolution. Here a main difficulty is how to find a path
which maximizes this difference, for the situation that the space of
the path is not compact. In the 1st version, in order to get ride of
this difficulty, we introduced an approach of \textquotedblleft
frozenness\textquotedblright\ of the main course of the paths where
the maximization takes place. But to apply this frozen procedure, we
need to modify the definition of time-derivative (or horizontal
derivative) which is a heavy cost. In this new version, we have
improved this \textquotedblleft frozen method\textquotedblright\ to
a \textquotedblleft left frozen\textquotedblright\ one. Using this
we can find a desired maximum path without changing the original
Dupire's definition of the horizontal derivative. This method can be
applied to obtain a comparison principle for smooth solutions as
well as for viscosity solutions of 1st and 2nd order fully nonlinear
PPDE. But for the case of the viscosity solution of 2nd order PPDE,
in order to get the comparison principle, we need solutions to
satisfy Condition (\ref{u-v}), which is still to be improved.

This paper uses PDE methods and the results can have direct applications to
stochastic analysis, e.g., martingales under a fully nonlinear expectation,
called $G$-expectation, stochastic optimal controls, stochastic games,
nonlinear pricing and risk measuring, and backward SDE. Recently many people
are very interested in this new theory of path dependent PDE. \cite{EKTZ}%
(EKTZ2011) introduced a different stochastic approach to derive a maximum
principle for a type of quasilinear PPDE, and the corresponding Perron's
approach to get the existence.

The note is organized as follows: in the next section we mainly recall the
notion of space and time (or vertical and horizontal) derivatives of
functional of paths, borrowed from [Dupire2009]. In section 3 we will
introduce the \textquotedblleft left frozen maximization\textquotedblright%
\ approach to obtain the maximum principle for $\mathbb{C}^{1,2}$-solutions of
fully nonlinear PDE. In Section 4, we introduce the notion of viscosity
solution of fully nonlinear path-dependent PDE. Section 5 is devoted to prove
the maximum principle of these new PDE for viscosity solutions. Many important
properties of this PPDE, such as uniqueness, monotonicity, positive
homogeneity and convexity can be derived from this new maximum principle. It
also provides a new PDE formulation of $G$-expectations with random
coefficients $G$ (see [Nutz2010] for a formulation of stochastic calculus).

After the 1st version of this paper, Xiangdong Li told me about a different
formulation of a type of loop-dependent PDE introduced by [Polyakov1980] in
Gauge theory (see also Li's paper). It's relation with the present path
dependent PDE is also an interesting problem.

\section{Notations}

For vectors $x,y\in\mathbb{R}^{n}$, we denote their scalar product by
$\left\langle x,y\right\rangle $ and the Euclidean norm $\left\langle
x,x\right\rangle ^{1/2}$ by $|x|$. We also denote the linear space of $n\times
n$ symmetric matrices by $\mathbb{S}(n)$.

The following notations are mainly from [Dupire2009]. Let $T>0$ be fixed. For
each $t\in\lbrack0,T]$, we denote by $\Lambda_{t}$ the set of right
continuous, $\mathbb{R}^{d}$-valued functions on $[0,t]$, namely,
$s_{i}\downarrow s$ implies $\omega(s_{i})\rightarrow\omega(s)$, for each
$\omega\in\Lambda_{t}$.

For each $\omega\in\Lambda_{T}$ the value of $\omega$ at time $s\in
\lbrack0,T]$ is denoted by $\omega(s)$. Thus $\omega=\omega(s)_{0\leq s\leq
T}$ is a right continuous process on $[0,T]$ and its value at time $s$ is
$\omega(s)$. The path of $\omega$ up to time $t$ is denoted by $\omega_{t}$,
i.e., $\omega_{t}=\omega(s)_{0\leq s\leq t}\in\Lambda_{t}$. We denote
$\Lambda=\bigcup_{t\in\lbrack0,T]}\Lambda_{t}$. We also specifically write
\[
\omega_{t}=\omega(s)_{0\leq s\leq t}=(\omega(s)_{0\leq s<t},\omega(t))
\]
to indicate the terminal position $\omega(t)$ of $\omega_{t}$ which plays a
special role in this framework. For each $\omega_{t}\in\Lambda$ and
$x\in\mathbb{R}^{d}$ we denote
\[
\omega_{t}^{x}(s)=\left\{
\begin{array}
[c]{ll}%
\omega(s), & \text{if }0\leq s<t,\\
\omega(t)+x, & \text{if }s=t,
\end{array}
\ \right.  \omega_{t,\delta}(s)=\left\{
\begin{array}
[c]{cl}%
\omega(s), & \text{if }0\leq s<t,\\
\omega(t), & \text{if }t\leq s<t+\delta.
\end{array}
\ \right.
\]
Sometimes we denote $(\omega_{t})^{x}=\omega_{t}^{x}$, and $(\omega
_{t})_{t,\delta}=\omega_{t,\delta}$. We also denote
\[
(\omega_{t}^{x})_{t,\delta}(s)=\left\{
\begin{array}
[c]{ll}%
\omega(s), & \text{if }0\leq s<t,\\
\omega(t)+x, & \text{if }t\leq s<t+\delta.
\end{array}
\ \right.
\]
Let $\bar{\omega}_{\bar{t}}$, $\omega_{t}\in\Lambda_{\bar{Q}}$ be given with
$t\geq\bar{t}$, we denote $\bar{\omega}_{\bar{t}}\otimes\omega_{t}\in
\Lambda_{t}$ by
\[
\bar{\omega}_{\bar{t}}\otimes\omega_{t}:=\left\{
\begin{array}
[c]{ll}%
\bar{\omega}(s), & \text{if }0\leq s<\bar{t},\\
\omega(s), & \text{if }\bar{t}\leq s<t.
\end{array}
\ \right.
\]
For a given open subset $Q\subset\mathbb{R}^{d}$, we denote it boundary by
$\partial Q$ and $\bar{Q}=Q\cup\partial Q$.
\begin{align*}
\Lambda_{Q_{t}}  &  :=\{\omega_{t}\in\Lambda_{t}:\omega(s)\in Q,\ s\in
\lbrack0,t]\},\,\,\Lambda_{Q}:=\bigcup_{t\in\lbrack0,T)}\Lambda_{Q_{t}},\ \\
\ \ \Lambda_{\bar{Q}_{t}}  &  :=\{\omega_{t}\in\Lambda:\omega(s)\in
Q,\ s\in\lbrack0,t),\ \omega(t)\in\bar{Q}\},\ \ \ \Lambda_{\bar{Q}}%
:=\bigcup_{t\in\lbrack0,T]}\Lambda_{\bar{Q}_{t}},\\
\Lambda_{\partial Q}  &  :=\{\omega_{t}\in\Lambda:\omega(s)\in Q,\ s\in
\lbrack0,t),\ \omega(t)\in\partial Q\}\cup\Lambda_{\bar{Q}_{T}}.
\end{align*}

We are interested in path functions. A path function $u$ is a real function
defined $\Lambda_{\bar{Q}}$, i.e., $u:\Lambda_{\bar{Q}}\mapsto\mathbb{R}$.
This function $u=u(\omega_{t})$, $\omega_{t}\in\Lambda_{\bar{Q}}$ can be also
regarded as a family of real valued functions :
\[
u(\omega_{t})=u(t,\omega(s)_{0\leq s\leq t})=u(t,\omega(s)_{0\leq s<t}%
,\omega(t)):\omega_{t}\in{\Lambda}_{\bar{Q}_{t}},\ \ t\in\lbrack0,T].
\]

\begin{definition}
We define
\begin{align*}
USC_{\ast}(\Lambda_{\bar{Q}})  &  :=\{u:\Lambda\rightarrow\mathbb{R}\text{:
such that}\\
&  \text{(i) For each fixed }\omega_{\hat{t}}\in\Lambda\bar{u}(t,x):=u((\omega
_{\hat{t}}^{x})_{\hat{t},t})\text{, }\\
&  \ \ \ \ \text{is a USC}\text{-function of } (t,x+\omega(\hat{t}))\in
\lbrack0,T-\hat{t}]\times\bar{Q} ;\\
&  \text{(ii)\ For each }\omega_{t}\in\Lambda_{\bar{Q}}\text{ with }%
t_{i}\uparrow t,\ \ \limsup_{i\rightarrow\infty}u(\omega_{t_{i}})\leq\sup
_{x}u(\omega_{t}^{x})\}.
\end{align*}
We also denote $LSC_{\ast}(\Lambda_{\bar{Q}}):=\{-u|u\in USC_{\ast}%
(\Lambda_{\bar{Q}})\}$. A function $u\in USC_{\ast}(\Lambda_{\bar{Q}})$ (resp.
$u\in LSC_{\ast}(\Lambda_{\bar{Q}})$) is called an $\Lambda$-upper (resp.
$\Lambda$-lower) semi continuous function. $u\in C_{\ast}(\Lambda_{\bar{Q}%
}):=$ $USC_{\ast}(\Lambda_{\bar{Q}})\cap LSC_{\ast}(\Lambda_{\bar{Q}})$ is
called an $\Lambda$-continuous function.
\end{definition}

\begin{definition}
Let $u:\Lambda_{\bar{Q}}\mapsto\mathbb{R}$ and $\omega_{t}\in\Lambda_{Q}$ be
given. If there exists $p\in\mathbb{R}^{d}$ such that
\[
u(\omega_{t}^{x})=u(\omega_{t})+\left\langle p,x\right\rangle
+o(|x|),\ \ x+\omega(t)\in Q,
\]
then we say $u$ is (vertically) differentiable (in $x$) at $\omega_{t}$, and
denote $D_{x}u(\omega_{t})=p$. If moreover, there exists $A\in\mathbb{S}(d)$
such that
\[
u(\omega_{t}^{x})=u(\omega_{t})+\left\langle p,x\right\rangle +\frac{1}%
{2}\left\langle Ax,x\right\rangle +o(|x|^{2}),\ \ x+\omega(t)\in Q,
\]
then we say $u$ is (vertically) twice differentiable (in $x$) at $\omega_{t}$,
and denote $D_{xx}^{2}u(\omega_{t})=A$.
\end{definition}

\begin{definition}
Let $u:\Lambda_{\bar{Q}}\mapsto\mathbb{R}$ and $\omega_{t}\in\Lambda_{Q}$ be
given. If there exists $a\in\mathbb{R}$ such that
\[
u(\omega_{t,\delta})-u(\omega_{t})=a\delta+o(\delta),\ \ \ \
\]
then we say that $u(\omega_{t})$ is (horizontally) differentiable (in $t$) at
$\omega_{t}$ and denote $D_{t}u(\omega_{t})=a$.
\end{definition}

\begin{definition}
We define $\mathbb{C}^{1,0}(\Lambda_{\bar{Q}})$, the set of functions
$u:=u(\omega_{t})$, $\omega_{t}\in\Lambda_{\bar{Q}}$ for which $D_{t}%
u(\omega_{t})$ exists, for each $\omega_{t}\in\Lambda_{Q}$, $t\in\lbrack t,T)$
and $u((\omega_{t})^{x})_{t,s-t})$, $D_{t}u((\omega_{t})^{x})_{t,s-t})$ are
continuous functions of $(s,x+\omega(t))\in\lbrack t,T)\times Q$.
\end{definition}

\begin{definition}
We define $\mathbb{C}^{1,2}(\Lambda_{\bar{Q}})$ as the set of functions
$u:=u(\omega_{t})$:$\Lambda_{\bar{Q}}\mapsto\mathbb{R}$, for which
$\varphi(\omega_{t})=D_{t}u(\omega_{t})$, $D_{x}u(\omega_{t})$, $D_{xx}%
^{2}u(\omega_{t})$ exist for $\omega_{t}\in\Lambda_{Q}$ and such that each
$\varphi((\omega_{t})^{x})_{t,s-t})$ is a continuous function of
$(s,x+\omega(t))\in\lbrack t,T)\times Q$.
\end{definition}

\section{Comparison principle for $\mathbb{C}^{1,2}$-solution of
path-dependent PDE}

In this paper we consider the following problem of path-dependent PDE of
parabolic PDE. To find $u\in\mathbb{C}^{1,2}(\Lambda_{\bar{Q}})$ such that
\begin{equation}%
\begin{array}
[c]{cc}%
D_{t}u(\omega_{t})+ & G(\omega_{t},u(\omega_{t}),D_{x}u(\omega_{t}%
),D_{xx}u(\omega_{t}))=0,
\end{array}
\omega_{t}\in\Lambda_{Q},\ \ \ \ \ \label{C2PPDE}%
\end{equation}
with a Cauchy condition:
\begin{equation}
u(\omega_{t})=\Phi(\omega_{t}),\ \ \ \
\omega_{t}\in\Lambda_{\partial Q},
\end{equation}
where $G:\Lambda_{\bar{Q}}\times\mathbb{R}\times\mathbb{R}^{d}\times
\mathbb{S}(d)\mapsto\mathbb{R}$ and $\Phi:\Lambda_{\partial Q_{T}}%
\mapsto\mathbb{R}$ are given functions. We make the following assumption for
$G$

\begin{description}
\item[(H1)] For each $\omega_{t}\in\Lambda_{Q}$, $u,v\in\mathbb{R}$,
$p\in\mathbb{R}^{d}$ and $X,Y\in\mathbb{S}(d)$ such that $u\geq v$, $X\leq
Y$,
\[
G(\omega_{t},u,p,X)\leq G(\omega_{t},v,p,Y).
\]

\end{description}

\begin{lemma}
[Left frozen maximization]\label{maxLemma}Let $Q$ be a bounded open subset of
$\mathbb{R}^{d}$ and let $u\in USC_{\ast}(\bar{Q})$ be bounded from above.
Then, for each $\omega_{t_{0}}^{(0)}\in\Lambda_{\bar{Q}}$, there exists
$\bar{\omega}_{\bar{t}}\in\Lambda_{\bar{Q}}$, $\bar{t}\in\lbrack t_{0},T]$,
such that $\bar{\omega}_{\bar{t}}=\omega_{t_{0}}^{(0)}\otimes\bar{\omega
}_{\bar{t}}$, $u(\bar{\omega}_{\bar{t}})\geq u(\omega_{t_{0}}^{(0)})$, and
\begin{equation}
u(\bar{\omega}_{\bar{t}})=\sup_{\gamma_{t}\in\Lambda_{\bar{Q}},\ t\geq\bar
{t}.}u(\bar{\omega}_{\bar{t}}\otimes\gamma_{t}).\text{ } \label{rmax}%
\end{equation}

\end{lemma}

\begin{proof}
Without loss of generality, we can assume that $u(\omega_{t_{0}}^{(0)})\geq
u(\omega_{t_{0}}^{(0),x})$, for all $x$ such that $x+\omega^{(0)}(t_{0}%
)\in\bar{Q}$. We set $m_{0}:=u(\omega_{t_{0}}^{(0)})$ and%
\[
\bar{m}_{0}:=\sup_{\gamma_{t}\in\Lambda_{\bar{Q}},\ t\geq t_{0}}%
u(\omega_{t_{0}}^{(0)}\otimes\gamma_{t})\geq m_{0}.
\]
If $\bar{m}_{0}=m_{0}$ then we can take $\bar{\omega}_{\bar{t}}=\omega_{t_{0}%
}^{(0)}$ and finish the procedure. Otherwise there exists $\omega_{t_{1}%
}^{(1)}\in\Lambda_{\bar{Q}}$ with $t_{1}>t_{0}$, such that $\omega_{t_{1}%
}^{(1)}=$ $\omega_{t_{0}}^{(0)}\otimes\omega_{t_{1}}^{(1)}$ and
\[
m_{1}:=u(\omega_{t_{1}}^{(1)})\geq\frac{m_{0}+\bar{m}_{0}}{2}.
\]
We set%
\[
\bar{m}_{1}:=\sup_{\gamma_{t}\in\Lambda_{\bar{Q}},\ t\geq t_{1}}%
u(\omega_{t_{1}}^{(1)}\otimes\gamma_{t})\geq m_{1}.
\]
If $\bar{m}_{1}=m_{1}$ then we can take $\bar{\omega}_{\bar{t}}=\omega_{t_{1}%
}^{(1)}$ and finish the procedure. Otherwise we can find, for $i=2,3,\cdots$,
$\omega_{t_{i}}^{(i)}\in\Lambda_{\bar{Q}}$ with, $t_{i}>t_{i-1}$ such that
$\omega_{t_{i}}^{(i)}=$ $\omega_{t_{i-1}}^{(i-1)}\otimes\omega_{t_{i}}^{(i)}$,
$u(\omega_{t_{i}}^{(i)})\geq u(\omega_{t_{i}}^{(i),x})$, for all $x$ such that
$x+\omega^{(i)}(t_{i})\in\bar{Q}$ and
\begin{align*}
m_{i}  &  :=u(\omega_{t_{i}}^{(i)})\geq\frac{m_{i-1}+\bar{m}_{i-1}}{2},\\
\bar{m}_{i}  &  :=\sup_{\gamma_{t}\in\Lambda_{Q},\ t\geq t_{i}}u(\omega
_{t_{i}}^{(i)}\otimes\gamma_{t})\geq m_{i}.
\end{align*}
and continue this procedure till the first time when $\bar{m}_{i}=m_{i}$ and
then finish the proof by setting $\bar{\omega}_{\bar{t}}=\omega_{t_{i}}^{(i)}%
$. For the last and \textquotedblleft worst\textquotedblright\ case in which
$\bar{m}_{i}>m_{i}$, for all $i=0,1,2,\cdots$, we have $t_{i}\uparrow\bar
{t}\in\lbrack0,T]$. Then we can find $\bar{\omega}_{\bar{t}}\in\Lambda
_{\bar{Q}}$ such that $\bar{\omega}_{\bar{t}}=\omega_{t_{i}}^{(i)}\otimes
\bar{\omega}_{\bar{t}}$. Since $u\in USC_{\ast}(\Lambda_{\bar{Q}})$, we can
choose $\bar{\omega}(\bar{t})\in\bar{Q}$ such that $u(\bar{\omega}_{\bar{t}%
})\geq u(\bar{\omega}_{\bar{t}}^{x})$, for all $x$ such that $x+\bar{\omega
}(\bar{t})\in\bar{Q}$. Since
\[
\bar{m}_{i+1}-m_{i+1}\leq\bar{m}_{i}-\frac{\bar{m}_{i}+m_{i}}{2}=\frac{\bar
{m}_{i}-m_{i}}{2},
\]
thus there exists $\bar{m}\in(m_{0},\bar{m}_{0})$, such that $\bar{m}%
_{i}\downarrow\bar{m}$ and $m_{i}\uparrow\bar{m}$. Thus $\lim_{i\rightarrow
\infty}u(\omega_{t_{i}}^{(i)})=\lim_{i\rightarrow\infty}m_{i}\leq
u(\bar{\omega}_{\bar{t}})$. \ We can claim that (\ref{rmax}) holds for this
$\bar{\omega}_{\bar{t}}$. Indeed, otherwise there exist $\gamma_{t}\in
\Lambda_{\bar{Q}}$ and $\delta>0$ with $t>\bar{t}$ and $\gamma_{t}=\bar
{\omega}_{\bar{t}}\otimes\gamma_{t}$, such that
\[
u(\bar{\omega}_{\bar{t}}\otimes\gamma_{t})\geq u(\bar{\omega}_{\bar{t}%
})+\delta=\bar{m}+\delta,\text{ }%
\]
then the following contradiction is induced:
\[
u(\bar{\omega}_{\bar{t}}\otimes\gamma_{t})=u(\omega_{t_{i}}^{(i)}\otimes
\gamma_{t})\leq\bar{m}_{i}\rightarrow\bar{m}.
\]

The proof is complete.
\end{proof}

\begin{definition}
A function $u\in\mathbb{C}^{1,2}(\Lambda_{\bar{Q}})$ is called a
$\mathbb{C}^{1,2}$-solution of the path dependent PDE (\ref{C2PPDE}) if for
each $\omega_{t}\in\Lambda_{Q}$, $t\in\lbrack0,T)$, the equality
(\ref{C2PPDE}) is satisfied. $u$ is called a subsolution (resp. supersolution)
of (\ref{C2PPDE}) if the \textquotedblleft$=$\textquotedblright\ in
(\ref{C2PPDE}) is replace by \textquotedblleft$\geq$\textquotedblright\ (resp.
\textquotedblleft$\leq$\textquotedblright).
\end{definition}

\begin{remark}
The solution of classical PDE is a special case when $u(\omega_{t})=\bar
{u}(t,\omega(t))$, $\bar{u}\in C^{1,2}([0,T)\times\bar{Q})$. Since, for each
$\omega_{t}\in\Lambda_{Q}$ and $t\in(0,T)$,
\[
\partial_{t}\bar{u}(t,\omega(t))=D_{t}u(\omega_{t}),\ \ \ D_{x}\bar
{u}(t,\omega(t))=D_{x}u(\omega_{t}),\ \ D_{xx}^{2}\bar{u}(t,\omega
(t))=D_{xx}^{2}u(\omega_{t}),\
\]
thus $\bar{u}(t,x)$ is a classical solution of PDE.
\end{remark}

\begin{lemma}
\label{max}Let $u\in\mathbb{C}^{1,2}(\Lambda_{\bar{Q}})$ and $\hat{\omega
}_{\hat{t}}\in\Lambda_{Q}$, $\hat{t}\in\lbrack0,T)$ be given satisfying
\begin{equation}
u(\hat{\omega}_{\hat{t}})\geq u(\hat{\omega}_{\hat{t}}\otimes\omega
_{t}),\ \ \ \text{for all }\,\,\omega_{t}\in\Lambda_{{Q}},\ t\geq\hat{t}.
\label{umax}%
\end{equation}
Then we have
\begin{equation}
D_{t}u(\hat{\omega}_{\hat{t}})\leq0,\ D_{x}u(\hat{\omega}_{\hat{t}%
})=0,\ \ D_{xx}^{2}u(\hat{\omega}_{\hat{t}})\leq0.\ \label{MAXtx}%
\end{equation}

\end{lemma}

\begin{proof}
We set $\omega(s)=x+\hat{\omega}(\hat{t})\mathbf{1}_{[\hat{t},t]}(s)$,
$s\in\lbrack0,t]$ and define
\[
\bar{u}(t,x):=u(\hat{\omega}_{\hat{t}}\otimes\omega_{t})=u((\hat{\omega}%
_{\hat{t}}^{x})_{\hat{t},t-\hat{t}}).
\]
For $x=0$, we derive from (\ref{umax}) condition that
\[
\bar{u}(\hat{t},0)\geq\bar{u}(t,0),\ \text{for }\ t\geq\hat{t},
\]
and thus $\frac{d}{dt}\bar{u}(\hat{t},0)\leq0$, or $D_{t}u(\hat{\omega}%
_{\hat{t}})\leq0$. For $t=\hat{t}$ we derive from (\ref{umax})\ $\bar{u}%
(\hat{t},0)\geq\bar{u}(\hat{t},x)$, for sufficiently small $x$, and thus
$D_{x}\bar{u}(\bar{t},0)=0$, $D_{xx}^{2}u(\hat{t},0)\leq0$, from which we have
the second and third relations of (\ref{MAXtx}).
\end{proof}

The following result is the so called comparison principle, or comparison
theorem, of PPDE for $\mathbb{C}^{1,2}$-solutions of (\ref{C2PPDE}).

\begin{theorem}
We make Assumption \textbf{(H1)}. Let $Q$ be a bounded open subset of
$\mathbb{R}^{d}$ and $u\in\mathbb{C}^{1,2}(\Lambda_{\bar{Q}})\cap USC_{\ast
}(\Lambda_{\bar{Q}})$ be a subsolution and $v\in\mathbb{C}^{1,2}(\Lambda
_{\bar{Q}})\cap LSC_{\ast}(\Lambda_{\bar{Q}})$ a supersolution of
(\ref{C2PPDE}). We also assume that $u-v$ is bounded from the above. Then the
maximum principle holds: if $u(\omega_{t})\leq v(\omega_{t})$ for all
$\omega_{t}\in\Lambda_{\partial Q}$, then we also have
\[
u(\omega_{t})\leq v(\omega_{t}),\ \ \ \forall\omega_{t}\in\Lambda_{Q}.\
\]

\end{theorem}

\begin{proof}
We observe that for $\bar{\delta}>0$, the function defined by $\tilde
{u}:=u-\bar{\delta}/t$ is a subsolution of
\[
D_{t}\tilde{u}(\omega_{t})+G(\omega_{t},\tilde{u}(\omega_{t}),D_{x}\tilde
{u}(\omega_{t}),D_{xx}\tilde{u}(\omega_{t}))\geq\frac{\bar{\delta}}{t^{2}}.
\]
Since $u\leq v$ follows from $\tilde{u}\leq v$ in the limit $\bar{\delta
}\downarrow0$, it suffices to prove the theorem under the additional
assumptions:
\begin{equation}%
\begin{array}
[c]{c}%
D_{t}u(\omega_{t})+G(\omega_{t},u(\omega_{t}),D_{x}u(\omega_{t}),D_{xx}%
^{2}u(\omega_{t}))\geq c,\ \ c:=\bar{\delta}/T^{2},\\
\ \ \ \ \ \ \ \ \ \ \ \ \ \text{and }\lim_{t\rightarrow0}u(\omega_{t}%
)=-\infty,\ \text{uniformly on }[0,T).
\end{array}
\label{inq-c}%
\end{equation}

Suppose by the contrary that there exists $\omega_{t_{0}}^{(0)}\in\Lambda_{Q}%
$, $t_{0}<T$, such that
\[
m_{0}:=u(\omega_{t_{0}}^{(0)})-v(\omega_{t_{0}}^{(0)})>0.
\]
Then, by Lemma \ref{maxLemma}, there exists $\bar{\omega}_{\bar{t}}\in
\Lambda_{Q}$ such that
\begin{equation}
u(\bar{\omega}_{\bar{t}})-v(\bar{\omega}_{\bar{t}})\geq\sup_{\gamma_{t}%
\in\Lambda_{Q},\ t\in\lbrack\bar{t},T]}u(\bar{\omega}_{\bar{t}}\otimes
\gamma_{t})-v(\bar{\omega}_{\bar{t}}\otimes\gamma_{t})\geq m_{0}.
\label{pathMax}%
\end{equation}
But by Lemma \ref{max},
\begin{align*}
D_{t}u(\bar{\omega}_{\bar{t}})-D_{t}v(\bar{\omega}_{\bar{t}})  &  \leq0,\\
D_{x}u(\bar{\omega}_{\bar{t}})-D_{x}v(\bar{\omega}_{\bar{t}})  &
=0,\ \ D_{xx}^{2}u(\bar{\omega}_{\bar{t}})-D_{xx}^{2}v(\bar{\omega}_{\bar{t}%
})\leq0.
\end{align*}
It follows that%
\begin{align*}
0  &  <c\leq D_{t}u(\bar{\omega}_{\bar{t}})+G(\bar{\omega}_{\bar{t}}%
,u(\bar{\omega}_{\bar{t}}),D_{x}u(\bar{\omega}_{\bar{t}}),D_{xx}^{2}%
u(\bar{\omega}_{\bar{t}}))\\
&  \leq D_{t}v(\bar{\omega}_{\bar{t}})+G(\bar{\omega}_{\bar{t}},v(\bar{\omega
}_{\bar{t}}),D_{x}v(\bar{\omega}_{\bar{t}}),D_{xx}^{2}v(\bar{\omega}_{\bar{t}%
}))\leq0.
\end{align*}
This induces a contradiction. The proof is complete.
\end{proof}

\section{Viscosity solutions for path-dependent PDE}

The notion viscosity solutions for classical (state-dependent) PDEs were
firstly introduced by Crandall and Lions [1981]. For many important
contributions in the developments of this powerful and elegant theory and rich
literature, we refer to the well-known user's guide by Crandall, Ishii and
Lions [1992]. A parabolic version of this theory quite helpful to understand
the present framework of path-dependent PDE can be found in the Appendix C of [Peng2010a].

Consider the following path-dependence parabolic PDE: to find a function
$u=u(\omega_{t})\in$ $USC_{\ast}(\Lambda_{\bar{Q}})$ such that
\begin{equation}
\left\{
\begin{array}
[c]{l}%
D_{t}u(\omega_{t})+G(u(\omega_{t}),D_{x}u(\omega_{t}),D_{xx}^{2}u(\omega
_{t}))=0\ \text{on }\omega_{t}\in\Lambda_{Q},\\
\ u(\omega_{T})=\Phi(\omega_{T})\ \text{for}\
\omega_{T}\in\Lambda_{\partial Q},
\end{array}
\right.  \label{PE-G}%
\end{equation}
\ where $G:\mathbb{R}\times\mathbb{R}^{d}\times\mathbb{S}(d)\mapsto\mathbb{R}$
and $\Phi\in USC_{\ast}(\Lambda_{\partial Q}):\Lambda_{\partial Q_{T}%
}\mapsto\mathbb{R}$ are given functions. We make the following assumption:

\begin{description}
\item[(H2)] Suppose that $G\in\mathbb{C}(\mathbb{R}\times\mathbb{R}^{d}%
\times\mathbb{S}(d))$ is a continuous function satisfying the following
condition:
\begin{equation}
G(u,p,X)\geq G(v,p,Y)\ \ \text{whenever}\ X\geq Y,\ \ u\leq v. \label{DE}%
\end{equation}

\end{description}

We now generalize the definition of viscosity solutions to the situation of PPDE.

Let $u\in USC(\Lambda_{\bar{Q}})$, we denote by $P^{2,+}u(\omega_{t})$ (the
\textquotedblleft\textbf{parabolic superjet}\textquotedblright\ of $u$ at a
given $\omega_{t}\in\Lambda_{Q}$, $t\in\lbrack0,T)$) the set of triples
$(a,p,X)\in\mathbb{R}\times\mathbb{R}^{d}\times\mathbb{S}(d)$ such that%
\[
u((\omega_{t}^{x})_{t,\delta})\leq u(\omega_{t})+a\delta+\left\langle
p,x\right\rangle +\frac{1}{2}\left\langle Xx,x\right\rangle +o(\delta
+|x|^{2}),\ \
\]
for each $\delta\geq0$ and $x\in\mathbb{R}^{d}$ such that $\omega(t)+x\in Q$.
It is easy to check that if $\varphi\in\mathbb{C}^{1,2}(\Lambda_{Q})$
satisfies
\[
u(\omega_{t})=\varphi(\omega_{t}),\ \ u((\omega_{t}^{x})_{t,\delta}%
)\leq\varphi((\omega_{t}^{x})_{t+\delta}),
\]
then $(D_{t}\varphi(\omega_{t}),D_{x}\varphi(\omega_{t}),D_{xx}\varphi
(\omega_{t}))\in P^{2,+}u(\omega_{t})$.

We also denote%
\begin{align*}
\bar{P}^{2,+}u(\omega_{t})=  &  \{(a,p,X)\in\mathbb{R}\times\mathbb{R}%
^{d}\times\mathbb{S}(d):\exists(\omega_{t_{n}}^{n},a_{n},p_{n},X_{n})\\
&  \ \ \text{such that}\ (a_{n},p_{n},X_{n})\in P^{2,+}u(\omega_{t_{n}}%
^{n})\ \text{and}\ \\
&  \ \ (\omega_{t_{n}}^{n},u(\omega_{t_{n}}^{n}),a_{n},p_{n},X_{n}%
)\rightarrow(\omega_{t},u(\omega_{t}),a,p,X)\}.
\end{align*}
We define $P^{2,-}u(\omega_{t})$ (the \textquotedblleft\textbf{parabolic
subjet}\textquotedblright\ of $u$ at $\omega_{t}$) by $P^{2,-}u(\omega
_{t})=-P^{2,+}(-u)(\omega_{t})$ and $P^{2,-}u(\omega_{t})$ by $\bar{P}%
^{2,-}u(\omega_{t})=-\bar{P}^{2,+}(-u)(\omega_{t})$.

\begin{definition}
A viscosity subsolution (resp. supersolution) of (\ref{PE-G}) on
$\Lambda_{\bar{Q}}$ is a function $u\in USC_{\ast}(\Lambda_{\bar{Q}})$ such
that for each fixed $\omega_{t}\in\Lambda_{Q}$, $t\in\lbrack0,T)$ and for each
$(a,p,X)\in P^{2,+}u(\omega_{t})$, (resp. $(a,p,X)\in P^{2,-}u(\omega_{t})$)
we have\ {}%
\begin{equation}
a+G(u(\omega_{t}),p,X)\geq0\ \ \text{(resp. }\leq0\text{). }\ \label{eq2.4}%
\end{equation}
$u$ is a viscosity solution if it is both sub and supersolution.
\end{definition}

\begin{remark}
Since $G$ is a continuous function, thus (\ref{eq2.4}) holds also for
$(a,p,X)\in\bar{P}^{2,+}u(\omega_{t})$ (resp. $(a,p,X)\in\bar{P}^{2,-}%
u(\omega_{t})$).
\end{remark}

\begin{remark}
If $\psi\in\mathbb{C}^{1,2}(\Lambda_{Q})$, $\hat{\omega}_{\hat{t}}\in\Lambda$,
$\hat{t}\in(0,T)$, $\psi(\hat{\omega}_{\hat{t}})=u(\hat{\omega}_{\hat{t}})$
are such that
\begin{align*}
\psi(\hat{\omega}_{\hat{t}}\otimes\gamma_{t})  &  \geq u(\hat{\omega}_{\hat
{t}}\otimes\gamma_{t}),\ \ \text{(resp. }\psi(\hat{\omega}_{\hat{t}}%
\otimes\gamma_{t})\leq u(\hat{\omega}_{\hat{t}}\otimes\gamma_{t}%
)\text{)\ \ }\\
\forall\gamma_{t}  &  \in\Lambda_{Q},\ t\in\lbrack\hat{t},T],
\end{align*}
then we have $(D_{t}\psi(\hat{\omega}_{\hat{t}}),D\psi(\hat{\omega}_{\hat{t}%
}),D^{2}\psi(\hat{\omega}_{\hat{t}}))\in P^{2,+}u(\hat{\omega}_{\hat{t}})$
(resp. $P^{2,-}u(\hat{\omega}_{\hat{t}})$).
\end{remark}

\begin{remark}
If $G=G(u,p)$ then (\ref{PE-G}) becomes a first order path-dependent
PDE:
\begin{equation}
\left\{
\begin{array}
[c]{l}%
D_{t}u(\omega_{t})+G(u(\omega_{t}),D_{x}u(\omega_{t}))=0\ \text{on }\omega_{t}\in\Lambda_{Q},\\
\ u(\omega_{T})=\Phi(\omega_{T})\ \text{for}\
\omega_{T}\in\Lambda_{\partial Q},
\end{array}
\right.  %
\end{equation}
In this case we can similarly define $P^{1,+}u(\omega_{t})$ (resp.
$P^{1,-}u(\omega_{t})$), $\bar{P}^{1,+}u(\omega_{t})$ (resp.
$\bar{P}^{1,-}u(\omega_{t})$) and use them to give the notion  the
corresponding viscosity solution.
\end{remark}

\section{Comparison principle for viscosity solution of 2nd order path-dependent PDE}

In this section we show that the approach of left frozen maximization
introduced in the proof of the comparison principle for smooth solutions of
PPDE can be also applied for the situation of viscosity solutions.

The following lemma is from Theorem 8.2 of [CIL1992]. See also
[Cradall1989].

\begin{lemma}
\label{Thm-8.3}Let $u_{i}\in$USC$((0,T)\times Q)$ for $i=1,\cdots,k$ be given.
Let $\varphi$ be a function defined on $(0,T)\times Q^{\otimes k}$ such that
$(t,x_{1},\ldots,x_{k})\rightarrow\varphi(t,x_{1},\ldots,x_{k})$ is once
continuously differentiable in $t$ and twice continuously differentiable in
$(x_{1},\cdots,x_{k})\in Q^{\otimes k}$. Suppose that $\hat{t}\in\lbrack0,T)$,
$\hat{x}_{i}\in\mathbb{R}^{d}$ for $i=1,\cdots,k$ and
\[
w(t,x_{1},\cdots,x_{k}):=u_{1}(t,x_{1})+\cdots+u_{k}(t,x_{k})-\varphi
(t,x_{1},\cdots,x_{k})
\]%
\begin{equation}
w(t,x_{1},\cdots,x_{k})\leq w(\hat{t},\hat{x}_{1},\cdots,\hat{x}_{k}%
),\ \ t\in\lbrack\hat{t},T),\ \ x_{i}\in Q \label{eq8.4}%
\end{equation}
Assume, moreover, that there exists $r>0$ such that for every $M>0$ there
exists constant $C$ such that for $i=1,\cdots,k$,%
\begin{equation}%
\begin{array}
[c]{ll}
& b_{i}\leq C\text{ whenever \ }(b_{i},q_{i},X_{i})\in\mathcal{P}^{2,+}%
u_{i}(t,x_{i}),\\
& |x_{i}-\hat{x}_{i}|+|t-\hat{t}|\leq r\text{ and }|u_{i}(t,x_{i}%
)|+|q_{i}|+\left\Vert X_{i}\right\Vert \leq M.
\end{array}
\label{eq8.5}%
\end{equation}
Then, there exist $b_{1},\cdots,b_{k}\in\mathbb{R}$, such that \newline%
\textsl{(i)} $(b_{i},D_{x_{i}}\varphi(\hat{t},\hat{x}_{1},\cdots,\hat{x}%
_{k}),X_{i})\in\mathcal{\bar{P}}^{2,+}u_{i}(\hat{t},\hat{x}_{i}%
),\ \ i=1,\cdots,k,$\newline\textsl{(ii)} $b_{1}+\cdots+b_{k}\leq\partial
_{t}\varphi(\hat{t},\hat{x}_{1},\cdots,\hat{x}_{k}),$\newline\textsl{(iii)}
\[
-(\frac{1}{\varepsilon}+\left\Vert A\right\Vert )I\leq\left[
\begin{array}
[c]{ccc}%
X_{1} & \cdots & 0\\
\vdots & \ddots & \vdots\\
0 & \cdots & X_{k}%
\end{array}
\right]  \leq A+\varepsilon A^{2},
\]
\newline where $A=D_{x}^{2}\varphi(\hat{x})\in\mathbb{S}(kd)$.
\end{lemma}

\begin{theorem}
\textbf{\label{Thm-dom1}}(Comparison principle) We assume \textbf{(H2)}. Let
$Q$ be a bounded open subset of $\mathbb{R}^{d}$ and $u\in$\textrm{USC}%
$_{\ast}(\Lambda_{\bar{Q}})$ (resp. $v\in$\textrm{LSC}$_{\ast}(\Lambda
_{\bar{Q}})$) be a viscosity subsolution (resp. supersolution) of
\begin{equation}
D_{t}u+G(u(\omega_{t}),D_{x}u(\omega_{t}),D_{xx}^{2}u(\omega_{t}%
))=0.\ \ \label{visPDE0}%
\end{equation}
Assume that $u-v$ is bounded from above by $C$ and they satisfy
\begin{equation}
|u(\omega_{t})-u(\upsilon_{s})|\vee|v(\omega_{t})-v(\upsilon_{s})|\leq
\rho(d(\omega_{t},\upsilon_{s})), \label{u-Lip}%
\end{equation}
where $\rho:[0,\infty)\mapsto\mathbb{R}$ is a given continuous and increasing
function with $\rho(0)=0$ and
\[
d(\omega_{t},\upsilon_{s}):=|\omega(t)-\upsilon(s)|^{2}+\int_{0}^{t\vee
s}|\omega(r\wedge t)-\bar{\omega}(r\wedge s)|^{2}dr.
\]
We also assume that, for each $\hat{\omega}_{\hat{t}}\in\Lambda_{Q}$, the
functions
\[
\hat{u}(t,x):=u((\hat{\omega}_{\hat{t}}^{x})_{\hat{t},t-\hat{t}}),\ \ \hat
{v}(t,x):=v((\hat{\omega}_{\hat{t}}^{x})_{\hat{t},t-\hat{t}}),\ t\in
\lbrack\hat{t},T),\ x+\hat{\omega}(\hat{t})\in Q,\
\]
satisfy%
\begin{equation}
\bar{P}^{2,+}\hat{u}(\hat{t},0)\subset\bar{P}^{2,+}u(\hat{\omega}_{\hat{t}%
}),\ \ \ \ \bar{P}^{2,-}\hat{v}(\hat{t},0)\subset\bar{P}^{2,-}v(\hat{\omega
}_{\hat{t}}). \label{u-v}%
\end{equation}
Then the following comparison holds: if $u(\omega_{t})\leq v(\omega_{t})$ for
$\omega_{t}\in\Lambda_{\partial Q_{T}}$, then $u(\omega_{t})\leq v(\omega
_{t})$ for $\omega_{t}\in\Lambda_{Q}$.
\end{theorem}

We first observe that for $\bar{\delta}>0$, the functions defined by
$\tilde{u}_{i}:=u_{i}-\bar{\delta}/t$ is a subsolution of
\[
D_{t}\tilde{u}(\omega_{t})+G(\tilde{u}(\omega_{t}),D_{x}\tilde{u}(\omega
_{t}),D_{xx}\tilde{u}(\omega_{t}))=\frac{\bar{\delta}}{t^{2}},
\]
Since $u\leq v$ follows from $\tilde{u}\leq v$ in the limit $\bar{\delta
}\downarrow0$, it suffices to prove the theorem under the additional
assumptions:
\begin{equation}
D_{t}u(\omega_{t})+G(u(\omega_{t}),D_{x}u(\omega_{t}),D_{xx}^{2}u(\omega
_{t}))\geq c,\ \ c:=\bar{\delta}/T^{2}. \label{ineq-c}%
\end{equation}

To prove this theorem, for $\omega_{t}^{1},\omega_{t}^{2}\in\Lambda$, we set%
\begin{align*}
w_{\alpha}(\omega_{t}^{1},\omega_{t}^{2})  &  :=u(\omega_{t}^{1})-v(\omega
_{t}^{2})-\varphi_{\alpha}(\omega_{t}),\ \ \ \omega=(\omega^{1},\omega^{2}),\\
\varphi_{\alpha}(\omega_{t}^{1},\omega_{t}^{2})  &  :=\frac{\alpha}%
{2}\left\Vert \omega_{t}^{1}-\omega_{t}^{2}\right\Vert ^{2}\\
&  =\frac{\alpha}{2}|\omega^{1}(t)-\omega^{2}(t)|^{2}+\frac{\alpha}{2}\int%
_{0}^{t}|\omega^{1}(s)-\omega^{2}(s)|^{2}ds.
\end{align*}

\begin{proof}
[Proof of Theorem \ref{Thm-dom1}]The proof is still based on our
\textquotedblleft left frozen maximization\textquotedblright\ approach. To
prove the theorem, we assume to the contrary that there exists $\tilde{\omega
}_{\tilde{t}}\in\Lambda_{Q}$, $\tilde{t}\in(0,T)$, such that $\tilde
{m}:=u(\tilde{\omega}_{\tilde{t}})-v(\tilde{\omega}_{\tilde{t}})>0$. \ Let
$\alpha$ be a large number such that
\[
\frac{1}{\alpha}+\rho(\frac{2}{\alpha}(\frac{1}{\alpha}+2C-M_{\ast}))\leq
\frac{1}{2}(\tilde{m}\wedge c).
\]
We set%
\[
M_{\alpha}:=\sup_{\omega_{t}\in\Lambda_{Q\times Q}\ }w_{\alpha}(\omega
_{t})\geq M_{\ast}:=\sup_{\bar{\omega}_{t}\in\Lambda_{Q}\ }[u(\bar{\omega}%
_{t})-v(\bar{\omega}_{t})]\geq\tilde{m}.
\]
First let us check that if $\omega_{t}=(\omega_{t}^{1},\omega_{t}^{2}%
)\in\Lambda_{\bar{Q}\times\bar{Q}}$ satisfies $w_{\alpha}(\omega_{t})+\frac
{1}{\alpha}\geq M_{\alpha}$. Then we have
\begin{align*}
\frac{\alpha}{2}\left\Vert \omega_{t}^{1}-\omega_{t}^{2}\right\Vert ^{2}  &
\leq\frac{1}{\alpha}+u(\omega_{t}^{1})-v(\omega_{t}^{2})-M_{\alpha}\\
&  \leq\frac{1}{\alpha}+u(\omega_{t}^{1})-v(\omega_{t}^{2})-M_{\ast}\\
&  \leq\frac{1}{\alpha}+2C-M_{\ast}.
\end{align*}
But we also have
\begin{align*}
M_{\ast}  &  \leq\frac{1}{\alpha}+[u(\omega_{t}^{1})-u(\omega_{t}%
^{2})]+[u(\omega_{t}^{2})-v(\omega_{t}^{2})]-\frac{\alpha}{2}\left\Vert
\omega_{t}^{1}-\omega_{t}^{2}\right\Vert ^{2}\\
&  \leq\frac{1}{\alpha}+\rho(\left\Vert \omega_{t}^{1}-\omega_{t}%
^{2}\right\Vert ^{2})+[u(\omega_{t}^{2})-v(\omega_{t}^{2})]-\frac{\alpha}%
{2}\left\Vert \omega_{t}^{1}-\omega_{t}^{2}\right\Vert ^{2}\\
&  \leq\frac{1}{\alpha}+\rho(\frac{2}{\alpha}(\frac{1}{\alpha}+2C-M_{\ast
}))+M_{\ast}-\frac{\alpha}{2}\left\Vert \omega_{t}^{1}-\omega_{t}%
^{2}\right\Vert ^{2}.
\end{align*}
Thus
\[
\frac{\alpha}{2}\left\Vert \omega_{t}^{1}-\omega_{t}^{2}\right\Vert ^{2}%
\leq\frac{1}{\alpha}+\rho(\frac{2}{\alpha}(\frac{1}{\alpha}+2C-M_{\ast}%
))\leq\frac{c}{2}.
\]
\ Moreover $\omega_{t}^{1}$, $\omega_{t}^{2}\in\Lambda_{Q}$. Indeed if,
$\omega_{t}^{1}$ or $\omega_{t}^{2}\in\Lambda_{\partial Q}$, then we will
deduce the following contradiction:
\begin{align*}
\tilde{m}  &  \leq M_{\ast}\leq M_{\alpha}\\
&  \leq\frac{1}{\alpha}+[u(\omega_{t}^{1})-u(\omega_{t}^{2})]+u(\omega_{t}%
^{2})-v(\omega_{t}^{2})\\
&  \leq\frac{1}{\alpha}+[u(\omega_{t}^{1})-u(\omega_{t}^{2})]\\
&  \leq\frac{1}{\alpha}+\rho(\frac{2}{\alpha}(\frac{1}{\alpha}+2C-M_{\ast
}))\leq\frac{\tilde{m}}{2}.
\end{align*}

Now we can apply Lemma \ref{maxLemma} to find $\bar{\omega}_{\bar{t}}%
=(\bar{\omega}_{\bar{t}}^{1},\bar{\omega}_{\bar{t}}^{2})\in\Lambda_{Q\times
Q}$ satisfying $\bar{\omega}_{\bar{t}}=\omega_{t}\otimes\bar{\omega}_{\bar{t}%
}$, $w_{\alpha}(\bar{\omega}_{\bar{t}})\geq w_{\alpha}(\omega_{t})$ such that
\begin{equation}
w_{\alpha}(\bar{\omega}_{\bar{t}})\geq w_{\alpha}(\bar{\omega}_{\bar{t}%
}\otimes\gamma_{t}),\ \ \forall\gamma_{t}\in\Lambda_{Q\times Q},\ \ t\geq
\bar{t}. \label{u-v-w}%
\end{equation}
Since $w_{\alpha}(\bar{\omega}_{\bar{t}})+\frac{1}{\alpha}\geq w_{\alpha
}(\omega_{t})+\frac{1}{\alpha}\geq M_{\alpha}$, we still have%
\[
\frac{\alpha}{2}\left\Vert \bar{\omega}_{t}^{1}-\bar{\omega}_{t}%
^{2}\right\Vert ^{2}\leq\frac{c}{2}.\ \ \ \
\]

We just need to take
\[
\gamma_{t}^{1}(s)=x+\bar{\omega}^{1}(\bar{t})\mathbf{1}_{[\bar{t}%
,t]}(s),\ \ \gamma_{t}^{2}(s)=y+\bar{\omega}^{2}(\bar{t})\mathbf{1}_{[\bar
{t},t]}(s)
\]
in (\ref{u-v-w}) and define%
\begin{align*}
\bar{u}(t,x)  &  =u(\bar{\omega}_{\bar{t}}^{1}\otimes\gamma_{t}^{1}%
)=u((\bar{\omega}_{\bar{t}}^{1})^{x})_{\bar{t},t-\bar{t}}),\ \ \bar
{v}(t,y)=v(\bar{\omega}_{\bar{t}}^{2}\otimes\gamma_{t}^{2})=v((\bar{\omega
}_{\bar{t}}^{1})^{y})_{\bar{t},t-\bar{t}}).\\
\bar{\varphi}_{\alpha}(t,x,y)  &  =\varphi_{\alpha}(\bar{\omega}_{\bar{t}%
}\otimes\gamma_{t})=\varphi_{\alpha}((\bar{\omega}_{\bar{t}}^{1})^{x}%
)_{\bar{t},t-\bar{t}},(\bar{\omega}_{\bar{t}}^{2})^{y})_{\bar{t},t-\bar{t}}).
\end{align*}
Inequality (\ref{u-v-w}) becomes%
\begin{align*}
\bar{u}(\bar{t},0)-\bar{v}(\bar{t},0)-\bar{\varphi}_{\alpha}(\bar{t},0,0)  &
\geq\bar{u}(t,x)-\bar{v}(t,y)-\bar{\varphi}_{\alpha}(t,x,y),\\
t  &  \geq\bar{t},\ \ x+\bar{\omega}^{1}(\bar{t}),\ y+\bar{\omega}^{2}(\bar
{t})\in Q.
\end{align*}
We then apply Lemma \ref{Thm-8.3} and use Condition (\ref{u-v}) to obtain
that, for each $\varepsilon>0$, there exist $X_{1},X_{2}\in\mathbb{S}(d)$ such
that
\begin{align*}
(b_{1},p,X_{1})  &  \in\bar{P}^{2,+}\bar{u}(\bar{t},0)\subset\bar{P}%
^{2,+}u(\bar{\omega}_{\bar{t}}^{1}),\ \ \\
(b_{2},p,X_{2})  &  \in\bar{P}^{2,-}\bar{v}(\bar{t},0)\subset\bar{P}%
^{2,-}v(\bar{\omega}_{\bar{t}}^{2}),
\end{align*}
with $p=\partial_{x}\bar{\varphi}_{\alpha}(\bar{t},0,0)$ and%
\begin{equation}
-(\frac{1}{\varepsilon}+\left\Vert A\right\Vert )I_{2d}\leq\left(
\begin{array}
[c]{cc}%
X_{1} & 0\\
0 & -X_{2}%
\end{array}
\right)  \leq A+\varepsilon A^{2},\ \ \ b_{1}-b_{2}\leq\partial_{t}%
\bar{\varphi}(t,0,0), \label{X1X2}%
\end{equation}
where $A$ is explicitly given by%
\[
A=\alpha J_{2d},\;\text{where}\;J_{2d}=\left(
\begin{array}
[c]{cc}%
I & -I\\
-I & I
\end{array}
\right)  .
\]
The second inequality of the first relation in (\ref{X1X2}) implies $X_{1}\leq
X_{2}$, the second relation implies
\begin{align*}
b_{1}-b_{2}  &  \leq\partial_{t}\bar{\varphi}(t,0,0)\\
&  =\frac{\alpha}{2}|\bar{\omega}^{1}(\bar{t})-\bar{\omega}^{2}(\bar{t}%
)|^{2}\ \ \ \ (\leq\frac{c}{2}).
\end{align*}
Form these it follows that,
\[
b_{1}+G(u(\bar{\omega}_{\bar{t}}^{1}),p,X_{1})\geq c,\ \ b_{2}+G(v(\bar
{\omega}_{\bar{t}}^{2}),p,X_{2})\leq0.\ \
\]
This, together with condition (\ref{DE}) for $G$, it follows that
\begin{align*}
c  &  \leq b_{1}+G(u(\bar{\omega}_{\bar{t}}^{1}),p,X_{1})-[b_{2}%
+G(v(\bar{\omega}_{\bar{t}}^{2}),p,X_{2})]+\frac{\alpha}{2}|\bar{\omega}%
^{1}(\bar{t})-\bar{\omega}^{2}(\bar{t})|^{2}\\
&  \leq\frac{\alpha}{2}|\bar{\omega}^{1}(\bar{t})-\bar{\omega}^{2}(\bar
{t})|^{2}\leq\frac{c}{2},\ \
\end{align*}
which induces a contradiction.
\end{proof}

\begin{remark}
The left frozen maximization procedure introduced in this version can be also
applied to obtain the corresponding domination theorem discussed in the 1st
version of this note.
\end{remark}

\begin{corollary}
We assume \textbf{(H2)}. Let $u\in$\textrm{USC}$_{\ast}(\Lambda_{\bar{Q}}%
)\cap\mathbb{C}^{1,0}(\Lambda_{\bar{Q}})$ (resp. $v\in$\textrm{LSC}$_{\ast
}(\Lambda_{\bar{Q}})\cap\mathbb{C}^{1,0}(\Lambda_{\bar{Q}})$) be a viscosity
subsolution (resp. supersolution) of (\ref{visPDE0}) satisfying (\ref{u-Lip}).
Then the comparison principle holds.
\end{corollary}

The proof of this corollary is based on the above theorem the
following two lemmas.

\begin{lemma}
Let $\bar{u}\in C^{1,0}((0,T)\times\mathbb{R}^{d})$ and let $(p,X)\in\bar
{J}^{2,+}\bar{u}(\hat{t},\hat{y})$ for a given $(\hat{t},\hat{y})\in
\lbrack0,T)$, here $\hat{t}$ is regarded as a fixed variable. We have
$(\partial_{t}\bar{u}(\hat{t},\hat{y}),p,X)\in\bar{P}^{2,+}\bar{u}(\hat
{t},\hat{y})$
\end{lemma}

\begin{proof}
$(p,X)\in\bar{J}^{2,+}\bar{u}(\hat{t},\hat{y})$ means that there exists
$\hat{y}^{(j)}\rightarrow\hat{y}$ and $(p^{(j)},X^{(j)})\in J^{2,+}\bar
{u}(\hat{t},\hat{y}^{(j)})$ such that $\bar{u}(\hat{t},\hat{y}^{(j)}%
)\rightarrow\bar{u}(\hat{t},\hat{y})$ and $(p^{(j)},X^{(j)})\rightarrow(p,X)$.
From
\[
\bar{u}(\hat{t},y)\leq\bar{u}(\hat{t},\hat{y}^{(j)})+\left\langle
p^{(j)},y-\hat{y}^{(j)}\right\rangle +\left\langle X^{(j)}(y-\hat{y}%
^{(j)}),y-\hat{y}^{(j)}\right\rangle +o(|y-\hat{y}^{(j)}|^{2})
\]
we have%
\begin{align*}
&  \bar{u}(t,y)-\bar{u}(\hat{t},\hat{y}^{(j)})\\
&  =[\bar{u}(t,y)-\bar{u}(\hat{t},y)]+[\bar{u}(\hat{t},y)-\bar{u}(\hat{t}%
,\hat{y}^{(j)})]\\
&  \leq\int_{0}^{1}\partial_{t}\bar{u}(\hat{t}+\alpha(t-\hat{t}),y)d\alpha
\cdot(t-\hat{t})\\
&  +\left\langle p^{(j)},y-\hat{y}^{(j)}\right\rangle +\left\langle
X^{(j)}(y-\hat{y}^{(j)}),y-\hat{y}^{(j)}\right\rangle +o(|y-\hat{y}^{(j)}%
|^{2})\\
&  =\partial_{t}\bar{u}(\hat{t}^{(j)},\hat{y}^{(j)})(t-\hat{t}^{(j)}%
)+\left\langle p^{(j)},y-\hat{y}^{(j)}\right\rangle +\left\langle
X^{(j)}(y-\hat{y}^{(j)}),y-\hat{y}^{(j)}\right\rangle \\
&  +o(|y-\hat{y}^{(j)}|^{2}+|t-\hat{t}^{(j)}|).
\end{align*}
It follows that $(\partial_{t}\bar{u}(\hat{t},\hat{y}^{(j)}),p^{(j)}%
,X^{(j)})\in P^{2,+}\bar{u}(\hat{t},\hat{y}^{(j)})$, $j=1,2,\cdots$. Sine this
sequence converges to $(\partial_{t}\bar{u}(\hat{t},\hat{y}),p,X)$, we then
get the result.
\end{proof}

We extend this result to the following case for path-dependent functions:

\begin{lemma}
\label{A1}For a given function $u\in\mathbb{C}^{1,0}(\Lambda_{\bar{Q}})$ and a
fixed $\hat{\omega}_{\hat{t}}\in\Lambda$, $\hat{t}\in(0,T)$, we set $\bar
{u}(x):=u((\hat{\omega}_{\hat{t}})^{x})$, $x+\hat{\omega}(\hat{t})\in Q$, and
assume that
\begin{equation}
(p,X)\in\bar{J}^{2,+}\bar{u}(0). \label{a1}%
\end{equation}
Then $(D_{t}u(\hat{\omega}_{\hat{t}}),p,X)\in\bar{P}^{2,+}u(\hat{\omega}%
_{\hat{t}})$.
\end{lemma}

\begin{proof}
(\ref{a1}) means that there exists a sequence $\{\hat{y}^{(j)}\}_{j=1}%
^{\infty}$ in $Q$ such that $(p^{(j)},X^{(j)})\in J^{2,+}\bar{u}(\hat{y}%
^{(j)})$ and $\hat{y}^{(j)}\rightarrow0$, $\bar{u}(\hat{y}^{(j)}%
)\rightarrow\bar{u}(0)$, $(p^{(j)},X^{(j)})\rightarrow(p,X)$ as $j\rightarrow
\infty$. From
\[
\bar{u}(y)\leq\bar{u}(\hat{y}^{(j)})+\left\langle p^{(j)},y-\hat{y}%
^{(j)}\right\rangle +\left\langle X^{(j)}(y-\hat{y}^{(j)}),y-\hat{y}%
^{(j)}\right\rangle +o(|y-\hat{y}^{(j)}|^{2}),
\]
we have,
\begin{align*}
u(((  &  \hat{\omega}_{\hat{t}})^{y})_{\hat{t},t-\hat{t}})-u((\hat{\omega
}_{\hat{t}})^{\hat{y}^{(j)}})\\
&  =[u(((\hat{\omega}_{\hat{t}})^{y})_{\hat{t},t-\hat{t}})-u((\hat{\omega
}_{\hat{t}})^{y})]+[u((\hat{\omega}_{\hat{t}})^{y})-u((\hat{\omega}_{\hat{t}%
})^{\hat{y}^{(j)}})]\\
&  \leq\int_{0}^{1}D_{t}u(((\hat{\omega}_{\hat{t}})^{y})_{\hat{t}%
,\alpha(t-\hat{t})})d\alpha\cdot(t-\hat{t})\\
&  +\left\langle p^{(j)},y-\hat{y}^{(j)}\right\rangle +\left\langle
X^{(j)}(y-\hat{y}^{(j)}),y-\hat{y}^{(j)}\right\rangle +o(|y-\hat{y}^{(j)}%
|^{2})\\
&  =D_{t}u((\hat{\omega}_{\hat{t}})^{\hat{y}^{(j)}})(t-\hat{t}^{(j)}%
)+\left\langle p^{(j)},y-\hat{y}^{(j)}\right\rangle +\left\langle
X^{(j)}(y-\hat{y}^{(j)}),y-\hat{y}^{(j)}\right\rangle \\
&  +o(|y-\hat{y}^{(j)}|^{2}+|t-\hat{t}^{(j)}|).
\end{align*}
It follows that $(D_{t}u((\hat{\omega}_{\hat{t}})^{\hat{y}^{(j)}}%
),p^{(j)},X^{(j)})\in P^{2,+}u((\hat{\omega}_{\hat{t}})^{\hat{y}^{(j)}})$. But
we also have, as $j\rightarrow\infty$,
\[
(D_{t}u((\hat{\omega}_{\hat{t}})^{\hat{y}^{(j)}}),p^{(j)},X^{(j)}%
)\rightarrow(D_{t}u(\hat{\omega}_{\hat{t}}),p,X).
\]
Thus $(D_{t}u(\hat{\omega}_{\hat{t}}),p,X)\in\bar{P}^{2,+}u(\hat{\omega}%
_{\hat{t}})$. The proof is complete.
\end{proof}

\section{Comparison principle for viscosity solution of 1st order
path-dependent PDE}

We consider the comparison principle for viscosity solutions of 1st
order path-dependent PDE. The following lemma is a generalization
the left frozen maximization provided in Lemma \ref{maxLemma}.

\begin{lemma}
\label{maxLemma-a}Let $Q$ be a bounded open subset of $\mathbb{R}^{d}$ and let
$u\in USC_{\ast}(\bar{Q}\times\bar{Q})$ be bounded from above. Then for each
given $\omega_{t_{0}}^{(0)}$, $\upsilon_{s_{0}}^{(0)}\in\Lambda_{\bar{Q}}$,
there exist $\bar{\omega}_{\bar{t}}$, $\bar{\upsilon}_{\bar{s}}\in
\Lambda_{\bar{Q}}$, satisfying $\bar{t}\geq t$, $\bar{s}\geq s$,
$u(\omega_{t_{0}}^{(0)},\upsilon_{s_{0}}^{(0)})\leq u(\bar{\omega}_{\bar{t}%
},\bar{\upsilon}_{\bar{s}})$ and $\bar{\omega}_{\bar{t}}=\omega_{t_{0}}%
^{(0)}\otimes\bar{\omega}_{\bar{t}}$, $\bar{\upsilon}_{\bar{s}}=\upsilon
_{s_{0}}^{(0)}\otimes\bar{\upsilon}_{\bar{s}}$, such that
\begin{equation}
u(\bar{\omega}_{\bar{t}},\bar{\upsilon}_{\bar{s}})=\sup_{\substack{\gamma
_{t}\in\Lambda_{\bar{Q}},\ t\geq\bar{t},\\\eta_{s}\in\Lambda_{\bar{Q}}%
,\ s\geq\bar{s}}}u(\bar{\omega}_{\bar{t}}\otimes\gamma_{t},\bar{\upsilon
}_{\bar{s}}\otimes\eta_{s}).\text{ } \label{rmax-a}%
\end{equation}

\end{lemma}

\begin{proof}
Clearly $u(\omega_{t_{0}}^{(0),x},\upsilon_{s_{0}}^{(0),y})$ is an
USC-function of $x$ and $y$. Thus without loss of generality, we can assume
that $u(\omega_{t_{0}}^{(0)},\upsilon_{s_{0}}^{(0)})\geq u(\omega_{t_{0}%
}^{(0),x},\upsilon_{s_{0}}^{(0),y})$, for all $x$ and $y$ such that
$x+\omega^{(0)}(t_{0})$, $y+\upsilon^{(0)}(s_{0})\in\bar{Q}$. We set
$m_{0}:=u(\omega_{t_{0}}^{(0)},\upsilon_{s_{0}}^{(0)})$ and%
\[
\bar{m}_{0}:=\sup_{\substack{\gamma_{t}\in\Lambda_{\bar{Q}},\ t\geq
t_{0}\\\eta_{s}\in\Lambda_{\bar{Q}},\ s\geq s_{0}}}u(\omega_{t_{0}}%
^{(0)}\otimes\gamma_{t},\upsilon_{s_{0}}^{(0)}\otimes\eta_{s})\geq m_{0}.
\]
If $\bar{m}_{0}=m_{0}$ then we can take $\bar{\omega}_{\bar{t}}=\omega_{t_{0}%
}^{(0)}$, $\bar{\upsilon}_{\bar{s}}=\upsilon_{s_{0}}^{(0)}$ and finish the
procedure. Otherwise there exists $\omega_{t_{1}}^{(1)}\in\Lambda_{\bar{Q}}$
with $t_{1}\geq t_{0}$, $s_{1}\geq s_{0}$ and $t_{1}+s_{1}>t_{0}+s_{0}$,
\[
u(\omega_{t_{1}}^{(1)},\upsilon_{s_{1}}^{(1)})\geq u(\omega_{t_{1}}%
^{(1),x},\upsilon_{s_{1}}^{(1),y}),
\]
for all $x,y$ such that $x+\omega^{(1)}(t_{1})$, $y+\upsilon^{(1)}(s_{1}%
)\in\bar{Q}$, satisfying $\omega_{t_{1}}^{(1)}=$ $\omega_{t_{0}}^{(0)}%
\otimes\omega_{t_{1}}^{(1)}$, $\upsilon_{s_{1}}^{(1)}=\upsilon_{s_{0}}%
^{(0)}\otimes\upsilon_{s_{1}}^{(1)}$and
\[
m_{1}:=u(\omega_{t_{1}}^{(1)},\upsilon_{s_{1}}^{(1)})\geq\frac{m_{0}+\bar
{m}_{0}}{2}.
\]
We set%
\[
\bar{m}_{1}:=\sup_{\substack{\gamma_{t}\in\Lambda_{\bar{Q}},\ t\geq
t_{1}\\_{\eta_{s}\in\Lambda_{\bar{Q}},\ s\geq s_{1}}}}u(\omega_{t_{1}}%
^{(1)}\otimes\gamma_{t},\upsilon_{s_{1}}^{(1)}\otimes\eta_{s})\geq m_{1},
\]
If $\bar{m}_{1}=m_{1}$ then we can take $\bar{\omega}_{\bar{t}}=\omega_{t_{1}%
}^{(1)}$, $\bar{\upsilon}_{\bar{s}}=\upsilon_{s_{1}}^{(1)}$ and finish the
procedure. Otherwise we can continue this procedure to find, for
$i=2,3,\cdots,$ $\omega_{t_{i}}^{(i)},\upsilon_{s_{i}}^{(i)}\in\Lambda
_{\bar{Q}}$ with, $t_{i}\geq t_{i-1}$, $s_{i}\geq s_{i-1}$ and $u(\omega
_{t_{i}}^{(i)},\upsilon_{s_{i}}^{(i)})\geq u(\omega_{t_{i}}^{(i),x}%
,\upsilon_{s_{i}}^{(i),y})$, for all $x$, $y$ such that $x+\omega^{(i)}%
(t_{i})$, $y+\upsilon^{(i)}(s_{i})\in\bar{Q}$, such that $\omega_{t_{i}}%
^{(i)}=$ $\omega_{t_{i-1}}^{(i-1)}\otimes\omega_{t_{i}}^{(i)}$, $\upsilon
_{s_{i}}^{(i)}=\upsilon_{s_{i-1}}^{(i-1)}\otimes\upsilon_{s_{i}}^{(i)}$ and
\begin{align*}
m_{i}  &  :=u(\omega_{t_{i}}^{(i)},\upsilon_{s_{i}}^{(i)})\geq\frac
{m_{i-1}+\bar{m}_{i-1}}{2},\\
\bar{m}_{i}  &  :=\sup_{_{\substack{\gamma_{t}\in\Lambda_{\bar{Q}},\ t\geq
t_{i}\\_{\eta_{s}\in\Lambda_{\bar{Q}},\ s\geq s_{i}}}}}u(\omega_{t_{i}}%
^{(i)}\otimes\gamma_{t},\upsilon_{s_{i}}^{(i)}\otimes\eta_{s})\geq m_{i},
\end{align*}
and continue this procedure till the first time when $\bar{m}_{i}=m_{i}$ and
then finish the proof by setting $\bar{\omega}_{\bar{t}}=\omega_{t_{i}}^{(i)}%
$. For the last and \textquotedblleft worst\textquotedblright\ case in which
$\bar{m}_{i}>m_{i}$, for all $i=0,1,2,\cdots$, we have $s_{i}\uparrow\bar{s}$,
$t_{i}\uparrow\bar{t}$, such that $\bar{s}$, $\bar{t}\in\lbrack0,T]$ Then we
can find $\bar{\omega}_{\bar{t}}$, $\bar{\upsilon}_{\bar{t}}\in\Lambda
_{\bar{Q}}$ such that $\bar{\omega}_{\bar{t}}=\omega_{t_{i}}^{(i)}\otimes
\bar{\omega}_{\bar{t}}$, $\bar{\upsilon}_{\bar{s}}=\upsilon_{s_{i}}%
^{(i)}\otimes\bar{\upsilon}_{\bar{s}}$ and $u(\bar{\omega}_{\bar{t}}%
,\bar{\upsilon}_{\bar{s}})\geq u(\bar{\omega}_{\bar{t}}^{x},\bar{\upsilon
}_{\bar{s}}^{y})$, for all $x$, $y$ such that $x+\bar{\omega}(\bar{t})$,
$y+\bar{\upsilon}(\bar{s})\in\bar{Q}$. Since $u\in USC_{\ast}(\Lambda_{\bar
{Q}})$. \ Since
\[
\bar{m}_{i+1}-m_{i+1}\leq\bar{m}_{i}-\frac{\bar{m}_{i}+m_{i}}{2}=\frac{\bar
{m}_{i}-m_{i}}{2},
\]
thus there exists $\bar{m}\in(m_{0},\bar{m}_{0})$, such that $\bar{m}%
_{i}\downarrow\bar{m}$ and $m_{i}\uparrow\bar{m}$. We can claim that
(\ref{rmax-a}) holds for this $(\bar{\omega}_{\bar{t}},\bar{\upsilon}_{\bar
{s}})$. Indeed, otherwise there exists a $\eta_{s}$, $\gamma_{t}\in
\Lambda_{\bar{Q}}$ with $s\geq\bar{s}$, $t\geq\bar{t}$ and $\gamma_{t}%
=\bar{\omega}_{\bar{t}}\otimes\gamma_{t}$, $\eta_{s}=\bar{\upsilon}_{\bar{s}%
}\otimes\eta_{s}$ and a $\delta>0$ such that
\[
u(\bar{\omega}_{\bar{t}}\otimes\gamma_{t},\bar{\upsilon}_{\bar{s}}\otimes
\eta_{s})\geq u(\bar{\omega}_{\bar{t}},\bar{\upsilon}_{\bar{s}})+\delta
=\bar{m}+\delta,\text{ }%
\]
then the following contradiction is induced:
\[
u(\bar{\omega}_{\bar{t}}\otimes\gamma_{t},\bar{\upsilon}_{\bar{s}}\otimes
\eta_{s})=u(\omega_{t_{i}}^{(i)}\otimes\gamma_{t},\upsilon_{s_{i}}%
^{(i)}\otimes\eta_{s})\leq\bar{m}_{i}\rightarrow\bar{m}.
\]

The proof is complete.
\end{proof}

\begin{description}
\item[(H3)] We make the following assumption for each $u,v\in\mathbb{R}$,
$p\in\mathbb{R}^{d}$ such that $u\geq v$,
\[
G(u,p)\leq G(v,p).
\]

\end{description}

\begin{theorem}
\textbf{\label{Thm-dom1-a}}(Comparison principle) We assume (\textbf{H3)}. Let
$Q$ be a bounded open subset of $\mathbb{R}^{d}$ and $u\in$\textrm{USC}%
$_{\ast}(\Lambda_{\bar{Q}})$ (resp. $v\in$\textrm{LSC}$_{\ast}(\Lambda
_{\bar{Q}})$) be a viscosity subsolution (resp. supersolution) of
\begin{equation}
D_{t}u+G(u(\omega_{t}),D_{x}u(\omega_{t}))=0.\ \ \label{visPDE0a}%
\end{equation}
Assume that $u$, $-v$ are bounded from above by $C$ and they are
continuous in the following sense: there exists a constant
$\bar{a}>0$, such that, for each $\omega_{t}\in\Lambda_{Q}$,
$\gamma_{s}=\omega_{t}\otimes\gamma_{s}$,
$\bar{\gamma}_{\bar{s}}=\omega_{t}\otimes\bar{\gamma}_{\bar{s}}\in
\Lambda_{\bar{Q}}$ such that $s$,$\bar{s}\in\lbrack
t,(t+\bar{a})\wedge T]$,
\begin{align}
&  |u(\omega_{t}\otimes\gamma_{s})-u(\omega_{t}\otimes\bar{\gamma}_{\bar{s}%
})|\vee|v(\omega_{t}\otimes\gamma_{s})-v(\omega_{t}\otimes\bar{\gamma}%
_{\bar{s}})|\label{u-Lipa}\\
&  \leq\rho(d(\gamma_{s},\bar{\gamma}_{\bar{s}})),\nonumber
\end{align}
where
\[
d(\gamma_{s},\bar{\gamma}_{\bar{s}})=|\gamma(s)-\gamma(\bar{s})|^{2}%
+|s-\bar{s}|^{2}+\int_{0}^{s\wedge\hat{s}}(s\wedge\hat{s}-r)|\gamma
(r)-\bar{\gamma}(r)|^{2}dr
\]
and $\rho:[0,\infty)\mapsto\mathbb{R}$ is a given continuous and increasing
function with $\rho(0)=0$. Then the following comparison principle holds: if
$u(\omega_{t})\leq v(\omega_{t})$ for $\omega_{t}\in\Lambda_{\partial Q_{T}}$,
then $u(\omega_{t})\leq v(\omega_{t})$ for $\omega_{t}\in\Lambda_{Q}$.
\end{theorem}

We first observe that for $\bar{\delta}>0$, the functions defined by
$\tilde{u}_{i}:=u_{i}-\bar{\delta}/t$ is a subsolution of
\[
D_{t}\tilde{u}(\omega_{t})+G(\tilde{u}(\omega_{t}),D_{x}\tilde{u}(\omega
_{t}))=\frac{\bar{\delta}}{t^{2}},
\]
Since $u\leq v$ follows from $\tilde{u}\leq v$ in the limit $\bar{\delta
}\downarrow0$, it suffices to prove the theorem under the additional
assumptions:
\begin{equation}
D_{t}u(\omega_{t})+G(u(\omega_{t}),D_{x}u(\omega_{t}))\geq c,\ \ c:=\bar
{\delta}/T^{2}. \label{ineq-c-a}%
\end{equation}

To prove this theorem, for $\omega_{t},\upsilon_{s}\in\Lambda_{\bar{Q}}$, we
set%
\[
w_{\alpha}(\omega_{t},\upsilon_{s}):=u(\omega_{t})-v(\upsilon_{s}%
)-\frac{\alpha}{2}d(\omega_{t},\upsilon_{s}).
\]

\begin{proof}
[Proof of Theorem \ref{Thm-dom1-a}]The proof is still based on our
\textquotedblleft left frozen maximization\textquotedblright\ approach. We
only need to prove that $u(\omega_{t})\leq v(\omega_{t})$ for $\omega_{t}%
\in\Lambda_{\bar{Q}}$, $t\in\lbrack T-\bar{a},T)$, then repeat the smae
procedure for cases $[T-i\bar{a},T-(i-1)\bar{a}).$ To this end we assume to
the contrary that there exists $\tilde{\omega}_{\tilde{t}}\in\Lambda_{Q}$,
$\tilde{t}\in\lbrack T-\bar{a},T)$, such that $\tilde{m}:=u(\tilde{\omega
}_{\tilde{t}})-v(\tilde{\omega}_{\tilde{t}})>0$. \ Let $\alpha$ be a large
number such that
\[
\frac{1}{\alpha}+\rho(\frac{2}{\alpha}(\frac{1}{\alpha}+2C-M_{\ast}))\leq
\frac{1}{2}(\tilde{m}\wedge c).
\]
We set%
\[
M_{\alpha}:=\sup_{\substack{\gamma_{t},\eta_{s}\in\Lambda_{\bar{Q}}\\t,s\geq
T-\bar{a}\ }}w_{\alpha}(\tilde{\omega}_{\tilde{t}}\otimes\gamma_{t}%
,\tilde{\omega}_{\tilde{t}}\otimes\eta_{s})\geq M_{\ast}:=\sup
_{\substack{\gamma_{s}\in\Lambda_{\bar{Q}}\\s\geq T-\bar{a}\\\ }%
}[u(\tilde{\omega}_{\tilde{t}}\otimes\gamma_{s})-v(\tilde{\omega}_{\tilde{t}%
}\otimes\gamma_{s})]\geq\tilde{m}.
\]
We fix $\bar{\gamma}_{\bar{t}}$, $\bar{\eta}_{\bar{s}%
}\mathbf{\in}\Lambda_{\bar{Q}}$ satisfying
$\bar{\gamma}_{\bar{t}}=\tilde{\omega
}_{\tilde{t}}\otimes\bar{\gamma}_{\bar{t}}$, $\bar{\eta}_{\bar{s}}%
=\tilde{\omega}_{\tilde{t}}\otimes\bar{\eta}_{\bar{s}}$ and $w_{\alpha}%
(\bar{\gamma}_{\bar{t}},\bar{\eta}_{\bar{s}})+\frac{1}{\alpha}\geq M_{\alpha}%
$. We can check that
\begin{align*}
\frac{\alpha}{2}d(\bar{\gamma}_{\bar{t}},\bar{\eta}_{\bar{s}})  &  \leq
\frac{1}{\alpha}+u(\bar{\gamma}_{\bar{t}})-v(\bar{\eta}_{\bar{s}})-M_{\alpha
}\\
&  \leq\frac{1}{\alpha}+u(\bar{\gamma}_{\bar{t}})-v(\bar{\eta}_{\bar{s}%
})-M_{\ast}\\
&  \leq\frac{1}{\alpha}+2C-M_{\ast}.
\end{align*}
But we also have
\begin{align*}
M_{\ast}  &  \leq\frac{1}{\alpha}+[u(\bar{\gamma}_{\bar{t}})-u(\bar{\eta
}_{\bar{s}})]+[u(\bar{\eta}_{\bar{s}})-v(\bar{\eta}_{\bar{s}})]-\frac{\alpha
}{2}d(\bar{\gamma}_{\bar{t}},\bar{\eta}_{\bar{s}})\\
&  \leq\frac{1}{\alpha}+\rho(d(\bar{\gamma}_{\bar{t}},\bar{\eta}_{\bar{s}%
}))+[u(\bar{\eta}_{\bar{s}})-v(\bar{\eta}_{\bar{s}})]-\frac{\alpha}{2}%
d(\bar{\gamma}_{\bar{t}},\bar{\eta}_{\bar{s}})\\
&  \leq\frac{1}{\alpha}+\rho(\frac{2}{\alpha}(\frac{1}{\alpha}+2C-M_{\ast
}))+M_{\ast}-\frac{\alpha}{2}d(\bar{\gamma}_{\bar{t}},\bar{\eta}_{\bar{s}}).
\end{align*}
Thus
\[
\frac{\alpha}{2}d(\bar{\gamma}_{\bar{t}},\bar{\eta}_{\bar{s}})\leq\frac
{1}{\alpha}+\rho(\frac{2}{\alpha}(\frac{1}{\alpha}+2C-M_{\ast}))\leq\frac
{c}{2}.
\]
\ Moreover $\bar{\gamma}_{\bar{t}},\bar{\eta}_{\bar{s}}\in\Lambda_{Q}$. Indeed
if, say $\bar{\eta}_{\bar{s}}\in\Lambda_{\partial Q}$, then we will deduce the
following contradiction:
\begin{align*}
\tilde{m}  &  \leq M_{\ast}\leq M_{\alpha}\\
&  \leq\frac{1}{\alpha}+[u(\bar{\gamma}_{\bar{t}})-u(\bar{\eta}_{\bar{s}%
})]+u(\bar{\eta}_{\bar{s}})-v(\bar{\eta}_{\bar{s}})\\
&  \leq\frac{1}{\alpha}+[u(\bar{\gamma}_{\bar{t}})-u(\bar{\eta}_{\bar{s}})]\\
&  \leq\frac{1}{\alpha}+\rho(\frac{2}{\alpha}(\frac{1}{\alpha}+2C-M_{\ast
}))\leq\frac{\tilde{m}}{2}.
\end{align*}
Now we can apply Lemma \ref{maxLemma-a} to find $\hat{\omega}_{\hat{t}}%
,\hat{\upsilon}_{\hat{s}}\in\Lambda_{Q}$ satisfying $\hat{\omega}_{\hat{t}%
}=\bar{\gamma}_{\bar{t}}\otimes\hat{\omega}_{\hat{t}}$ and $\hat{\upsilon
}_{\hat{s}}=\bar{\eta}_{\bar{s}}\otimes\hat{\upsilon}_{\hat{s}}$ with
$w_{\alpha}(\bar{\omega}_{\bar{t}},\bar{\upsilon}_{\bar{s}})\geq w_{\alpha
}(\bar{\gamma}_{\bar{t}},\bar{\eta}_{\bar{s}})$ such that
\begin{align}
w_{\alpha}(\hat{\omega}_{\hat{t}},\hat{\upsilon}_{\hat{s}})  &  \geq
w_{\alpha}(\hat{\omega}_{\hat{t}}\otimes\gamma_{t},\hat{\upsilon}_{\hat{s}%
}\otimes\eta_{s}),\ \ \label{u-v-w-a}\\
\forall\gamma_{t},\eta_{s}  &  \in\Lambda_{Q},\ \ t\geq\hat{t},\ s\geq\hat
{s}.\nonumber
\end{align}
Since $w_{\alpha}(\hat{\omega}_{\hat{t}},\hat{\upsilon}_{\hat{s}})+\frac
{1}{\alpha}\geq w_{\alpha}(\bar{\gamma}_{\bar{t}},\bar{\eta}_{\bar{s}}%
)+\frac{1}{\alpha}\geq M_{\alpha}$, we still have%
\[
\frac{\alpha}{2}d(\hat{\omega}_{\hat{t}},\hat{\upsilon}_{\hat{s}})\leq\frac
{c}{2},\ \ \ \
\]
as well as
\[
u(\hat{\omega}_{\hat{t}})-v(\hat{\upsilon}_{\hat{s}})\geq M_{\alpha}-\frac
{1}{\alpha}\geq0.
\]

We just need to take
\begin{align*}
\gamma_{t}(r)  &  =x+\hat{\omega}(\hat{t})\mathbf{1}_{[\hat{t},t]}%
(r),\ \ r\leq t,\ \\
\eta_{s}(r)  &  =y+\hat{\upsilon}(\hat{s})\mathbf{1}_{[\hat{s},s]}%
(r),\ \ r\leq s,
\end{align*}
in (\ref{u-v-w-a}) and define%
\begin{align*}
\bar{u}(t,x)  &  =u(\hat{\omega}_{\hat{t}}\otimes\gamma_{t})=u((\hat{\omega
}_{\hat{t}})^{x})_{\hat{t},t-\hat{t}}),\ \ \\
\bar{v}(s,y) &  =v(\hat{\upsilon}_{\hat{s}}\otimes\eta_{s})=v((\hat{\upsilon
}_{\hat{s}})^{y})_{\hat{s},s-\hat{s}}),\\
\bar{\varphi}_{\alpha}(t,s,x,y)  &  =\varphi_{\alpha}(\hat{\omega}_{\hat{t}%
}\otimes\gamma_{t},\hat{\upsilon}_{\hat{s}}\otimes\eta_{s})=\varphi_{\alpha
}((\hat{\omega}_{\hat{t}})^{x})_{\hat{t},t-\hat{t}},(\hat{\upsilon}_{\hat{s}%
})^{y})_{\hat{s},s-\hat{s}}).
\end{align*}
Inequality (\ref{u-v-w-a}) becomes%
\begin{align*}
\bar{u}(\hat{t},0)-\bar{v}(\hat{t},0)-\bar{\varphi}_{\alpha}(\hat{t},0,0)  &
\geq\bar{u}(t,x)-\bar{v}(s,y)-\bar{\varphi}_{\alpha}(t,s,x,y),\\
t  &  \geq\hat{t},\ s\geq\hat{s},\ \ x+\hat{\omega}(\hat{t}),\ y+\hat
{\upsilon}(\hat{s})\in Q.
\end{align*}
Since
\[
d(\omega_{t},\upsilon_{s}):=|t-s|^{2}+|\omega(t)-\upsilon(s)|^{2}+\int%
_{0}^{t\wedge s}(t\wedge s-r)|\omega(r)-\upsilon(r)|^{2}dr,
\]
thus
\begin{align*}
b  &  :=\partial_{t}\bar{\varphi}_{\alpha}(\hat{t},\hat{s},0,0)=\alpha(\hat
{t}-\hat{s})=-\partial_{s}\bar{\varphi}_{\alpha}(\hat{t},\hat{s},0,0),\ \ \\
p  &  :=\partial_{x}\bar{\varphi}_{\alpha}(\hat{t},\hat{s},0,0)=\alpha
(\hat{\omega}(\hat{t})-\hat{\upsilon}(\hat{s}))=-\partial_{y}\bar{\varphi
}_{\alpha}(\hat{t},\hat{s},0,0).
\end{align*}
It follows that
\begin{align*}
(b,p)  &  \in P^{1,+}\hat{u}(\hat{t},0)\subset P^{1,+}u(\hat{\omega}_{\hat{t}%
}),\ \ \\
(b,p)  &  \in P^{1,-}\hat{v}(\hat{s},0)\subset\bar{P}^{1,-}v(\hat{\upsilon
}_{\hat{s}}).
\end{align*}
Consequently
\[
b+G(u(\hat{\omega}_{\hat{t}}),p)\geq c,\ \ b+G(v(\hat{\upsilon}_{\hat{s}%
}),p)\leq0.\ \
\]
It follows that
\[
c\leq b+G(u(\hat{\omega}_{\hat{t}}),p)-[b+G(v(\hat{\upsilon}_{\hat{s}%
}),p)]\leq0,\ \
\]
which induces a contradiction.
\end{proof}

\begin{remark} The above method can be also applied to obtain a
comparison principle for the case $G=G(\omega_t,u,p)$, $\omega_t\in
\Lambda_{\bar{Q}}$ under the following condition:
\[
G(\omega_{t}\otimes\gamma_{s},u,p)- G(\omega_{t}\otimes\bar{\gamma}_{\bar{s}%
},v,p)\leq \rho(d(\gamma_{s},\bar{\gamma}_{\bar{s}})),
\]
for each $u,v\in\mathbb{R}$, $p\in\mathbb{R}^{d}$,
 such that $u\geq v$ and for each
$\omega_{t}\in\Lambda_{Q}$,
$\gamma_{s}=\omega_{t}\otimes\gamma_{s}$,
$\bar{\gamma}_{\bar{s}}=\omega_{t}\otimes\bar{\gamma}_{\bar{s}}\in
\Lambda_{\bar{Q}}$ such that $s$,$\bar{s}\in\lbrack
t,(t+\bar{a})\wedge T]$.

\end{remark}

\end{document}